\documentclass[12pt]{article}
\usepackage[utf8]{inputenc}
\usepackage{geometry}
\usepackage{epsfig}
\usepackage{oldgerm}
\usepackage{amsmath}
\usepackage{amssymb}
\usepackage{amsthm}
\usepackage{cite}
\usepackage{enumerate}
\usepackage{amstext,color}
\usepackage{fullpage}
\usepackage{graphicx,graphics}
\usepackage{graphicx}
\usepackage{algpseudocode} 
\usepackage{algorithm} 
\usepackage{float}
\usepackage{subcaption}

\usepackage[caption = false]{subfig}

\newcommand{\x}[0]{\boldsymbol{x}}

\newtheorem{theorem}{Theorem}[section]

\newtheorem{lemma}{Lemma}[section]

\newtheorem{remark}{Remark}[section]

\newcommand{\nw}[1]{\vert\vert #1\vert\vert_W}
\newcommand{\nv}[1]{\vert\vert #1\vert\vert_V}
\newcommand{\nvh}[1]{\vert\vert #1\vert\vert_{V_h}}
\newcommand{\intq}{\int_{0}^{T}\int_{\Omega}}
\newcommand{\dx}{\ d\Omega dt}
\newcommand{\wh}[1]{\widehat{#1}}
\usepackage{ragged2e}
\justifying
\begin{document}
	\begin{center}
		{{\bf \large {\rm {\bf Space-time isogeometric method for a linear fourth order time dependent problem}}}}
	\end{center}
	\begin{center}

		{\textmd {{\bf Shreya Chauhan*}}}\footnote{\it *Corresponding Author, Department of Basic Sciences, Institute of Infrastructure, Technology, Research and Management, Gujarat, India, (shreyachauhan898@gmail.com)}
		{\textmd {\bf Sudhakar Chaudhary}}\footnote{\it Department of Basic Sciences, Institute of Infrastructure, Technology, Research and Management, Gujarat, India, (dr.sudhakarchaudhary@gmail.com)}
		\end{center}
	\begin{abstract}
		This article focuses on the space-time isogeometric method for a linear time dependent fourth order problem. Using an auxiliary variable, first the problem is split into a system of two second order differential equations and then the system is discretized by employing the tensor product spline spaces of time and spatial variables. We use the Babu\u{s}ka's theorem to prove the well-posedness of the continuous variational formulation. Also, the inf-sup stability condition at discrete level is established, which we use to prove the error estimates for the proposed method. Finally, to demonstrate the convergence of the scheme, few numerical results are reported. 
	\end{abstract}
	\noindent	{\bf Keywords:}  Fourth order time dependent problem, Space-time method, Isogeometric analysis,  Error estimates\\\\
\noindent {\bf AMS(MOS):} 65M60, 65D07, 65M12, 65M15

\section{Introduction}
In this article, we consider the following linear fourth order equation: 
\begin{equation}\label{3eqnmain}
	\begin{split}
		\partial_t u+\Delta^2u-\Delta u&=f\quad\mbox{in }Q_T,\\
		u(\x,0)&=0\quad\mbox{in }\Omega,\\
        u=\Delta u&=0\quad\mbox{on }\Sigma_T,
	\end{split}
\end{equation}
where $Q_T=\Omega\times(0,T]$, $\Sigma_T=\partial\Omega\times(0,T]$, $\Omega\subset \mathbb{R}^2$ is a convex domain with polygonal boundary $\partial\Omega$, $f:\Omega\times(0,T)\rightarrow\mathbb R$ is a given function and $T>0$. Equation \eqref{3eqnmain} can be considered as a particular case of linearized version of the Extended Fisher Kolmogorav (EFK) equation \cite{app3}. 
\par Fourth order time dependent problems appear in the modeling of various physical and engineering phenomena, for details see \cite{app1,app3}. Due to its wide range of real world applications, such problems have attracted the attention of many researchers. In \cite{Li2005}, the mixed Finite Element Method (FEM) for the fourth order linear elliptic and parabolic problems was studied using the radial basis functions. A local discontinuous Galerkin (dG) method for linear fourth order time dependent problems was analyzed in \cite{dong2009analysis}. The use of local dG method for discretization leads to an increase in the number of unknowns in obtaining the approximate solution. As an another approach, authors in \cite{liu2018mixed} introduced a penalty free mixed dG method for fourth order problems by considering the mixed formulation of the problem. Recently, in \cite{gu2025stabilizer} the fourth order linear parabolic problem is solved numerically by combining a stabilizer free weak Galerkin method with implicit $\theta$-scheme in time. Danumjaya et al. \cite{danumjaya2012mixed} analyzed the mixed FEM for the nonlinear EFK equation. Das et al. \cite{DAS202452} proposed a unified mixed method employing the biorthogonal basis functions to obtain the approximate solution of more general nonlinear EFK equation with different boundary conditions. Apart from the  conforming FEM, several other methods such as nonconforming FEM \cite{PEI20181383}, discontinuous Galerkin method \cite{dgmethod}, $C^0$ interior penalty method \cite{GUDI20131} have also been considered for the numerical treatment of nonlinear fourth order problems. In contrast to the existing works on fourth order time dependent problems, in this work the isogeometric discretization of the space-time variational formulation of \eqref{3eqnmain} is considered.
\par The Isogeometric Analysis (IgA) is a numerical scheme introduced by Hughes et al. \cite{IgAIntro} with the aim to enhance the compatibility between Computer Aided Design (CAD) and Finite Element Analysis (FEA). For more details and advantages of IgA, we refer to the monograph \cite{IgAbook}. The main idea behind IgA is to use the same basis functions to represent the computational domain and to approximate the solution to a partial differential equation. The most commonly used basis functions are B-splines and Non Uniform Rational B-Splines (NURBS). The use of splines with high-continuity provides more accurate solution compared to the standard $C^0$ piecewise-polynomial approximation \cite{Bressan2019}. 
\par The basis functions used in IgA can be globally $C^1$ smooth making it a suitable choice for obtaining numerical approximation of the fourth order problems in the frame work of standard Galerkin method \cite{TAGLIABUE2014277}. In \cite{MOORE2018673}, the discontinuous Galerkin IgA is presented for the fourth order elliptic problem on domains having non-overlapping patches. The mixed isogeometric method for spectral approximation of higher order differential operators is analyzed in \cite{DENG2019297}. The fourth order Cahn-Hilliard equation in IgA framework was first studied by G\'omez et al. \cite{GOMEZ20084333}. Since this seminal work, the isogeometric discretization with various temporal discretization methods have been considered for the Cahn-Hilliard equation \cite{KAESSMAIR2016783,KASTNER2016360,ZHANG2019112569}. Recently in \cite{Meng2024}, the authors provided the theoretical convergence results of time-stepping IgA for Cahn-Hilliard equation. 
\par The standard approach for solving the time dependent problems numerically is to first discretize the spatial domain using spatial discretization techniques like FEM, IGA etc and then to discretize the temporal domain using time-stepping methods  or vice-versa. In such methods, the space and time variable are treated distinctively and usually different discretizations are applied to each. This leads to a sequential procedure which requires parallelization in time \cite{50years}. An alternate approach is to consider the simultaneous discretization in space and time, that is the space-time methods. In space-time methods, the time variable is treated as just another spatial variable and the entire space-time cylinder is discretize. The space-time methods have advantages over the time-stepping methods, specifically in the adaptive discretization for refining in space and time simultaneously and in the treatment of problems with moving domains.  We refer to \cite{Steinbachreview} for an comprehensive overview of various space-time methods. Using the space-time formulation in IgA, we can exploit the properties of smooth basis function in time as well. In \cite{LANGER2016342}, the space-time isogeometric analysis for parabolic evolution equations was first introduced by employing the time upwind test functions to produce coercive bilinear form with respect to a discrete norm. Authors in \cite{LOLI20202586} proposed space-time isogeometric method for parabolic problems considering the space-time variational formulation in Bochner spaces. Therein authors also provided an efficient solver with robust preconditioner to solve the linear system.
The space-time least square isogeometric method with efficient solver is analyzed in \cite{montardini2020space}. In \cite{saade2021space,antonietti2023,chaudhary2024space}, the non-linear time dependent problems are analyzed using space-time isogeometric method.
\par Motivated from above, in this work we analyze the space-time mixed isogeometric method for problem \eqref{3eqnmain}. To the best of our knowledge, this is the first attempt to apply the space-time isogeometric method to the time dependent fourth order equation (linear EFK equation). To this end, we split problem \eqref{3eqnmain} into an equivalent system of second order equations by setting $-\Delta u=v$ as follows:
\begin{equation}\label{3pro1}
	\begin{split}
		\partial_tu-\Delta v+v&=f\quad\mbox{in }Q_T,\\
		-\Delta u&=v\quad\mbox{in }Q_T,\\
		u(\x,0)&=0\quad\mbox{in }\Omega,\\
		u=v&=0\quad\mbox{on }\Sigma_T.
	\end{split}
\end{equation}
This decoupling allows the use of $C^0$ finite element spaces to approximate the solution of \eqref{3eqnmain}. Inspired from \cite{stfempara}, we consider the space-time variational formulation of the system \eqref{3pro1}. As far as we know, the existence-uniqueness and error estimate results of the mixed space-time IgA for problem \eqref{3eqnmain} are totally new.
\par The remainder of the paper is outlined as follows: We introduce the space-time variational formulation for the problem \eqref{3pro1} and discuss the well-posedness results in Section \ref{3sec3stweak}. In Section \ref{3sec4dis}, the space-time isogeometric discretization of \eqref{3pro1} is developed and the unique solvability results with the corresponding error estimates of the scheme are presented. Few numerical experiments are provided to support the theoretical convergence results in Section \ref{3sec5numerics}.  Finally, in Section \ref{3sec5} we give some concluding remarks.\\


\noindent {\bf Notation.}  Let $X$ be a Banach space and $H^1(\Omega)$ be the standard Sobolev space. By $H_0^1(\Omega)$, we denote the space of all $H^1(\Omega)$ functions vanishing on $\partial\Omega$ and the dual of $H_0^1(\Omega)$ is denoted by $H^{-1}(\Omega)$. We denote by $L^2(X)$ the Bochner space defined by $L^2(X)=\{w: (0,T)\rightarrow X : \int_{0}^{T}||w||_X^2 < \infty \} .$  Let $V=\{w\in L^2(H_0^1(\Omega)):{\partial_t w\in L^2(H^{-1}(\Omega))}, w(\x,0)=0\quad\mbox{in }\Omega\}$, $W=L^2(H_0^1(\Omega))$ and $W^*=L^2(H^{-1}(\Omega))$. The norms on the spaces $V$ and $W$ are as follows:\\
\begin{equation*}
    \vert\vert v\vert\vert_V^2=\vert\vert\partial_tu\vert\vert_{L^2(H^{-1}(\Omega))}^2+\vert\vert u\vert\vert^2_{L^2(H_0^1(\Omega))}\mbox{ and }\vert\vert w\vert\vert_W=\vert\vert w\vert\vert_{L^2(H_0^1(\Omega))}.
\end{equation*}
We also introduce the product spaces $\mathcal{V}=V\times W$ and $\mathcal{W}=W\times W$ . The norms on these product spaces are as follows: 
\begin{equation*}
    \vert\vert(u,v)\vert\vert^2_{\mathcal{V}}=\vert\vert u\vert\vert^2_V+\vert\vert v\vert\vert^2_W\mbox{ and }\vert\vert (\phi,\psi)\vert\vert^2_\mathcal{W}=\vert\vert \phi\vert\vert^2_W+\vert\vert \psi\vert\vert^2_W.
\end{equation*}
Throughout the paper, $C$ denotes a generic constant which does not depend on the discretization parameters. 
\section{Space-time weak formulation and well-posedness results}\label{3sec3stweak}


In this section, we begin with the space-time weak formulation of the system of equations \eqref{3pro1} and recall some known results. Then these results will be used to show the existence-uniqueness of solution to space-time weak formulation.\\
Given $f\in W^*$, the space time formulation is to find $u\in V$ and $v\in W$ such that
\begin{equation}\label{3pro1st}
	\begin{split}
		\int_{0}^{T}\int_{\Omega}\partial_tu\phi\ d\Omega dt&+\int_{0}^{T}\int_{\Omega}\nabla v\nabla\phi\ d\Omega dt+\int_{0}^{T}\int_{\Omega}v\phi\ d\Omega dt\\
		&=\int_{0}^{T}\int_{\Omega}f\phi\ d\Omega dt\quad\forall \phi\in W,\\
		\int_{0}^{T}\int_{\Omega}\nabla u\nabla \psi\ d\Omega dt&=\int_{0}^{T}\int_{\Omega}v\psi\ d\Omega dt\quad\forall \psi\in W.
	\end{split}
\end{equation}
The following saddle point problem is an equivalent version of \eqref{3pro1st}: 
Find $(u,v)\in\mathcal{V}$ such that
\begin{equation}\label{3pro1bf}
	a(u,v;\phi,\psi)=F(\phi,\psi)\quad\forall(\phi,\psi)\in \mathcal{W},
\end{equation}
where
\begin{equation}\label{3bili}
	\begin{split}
	a(u,v;\phi,\psi)=&\int_{0}^{T}\int_{\Omega}\partial_tu\phi\ d\Omega dt+\int_{0}^{T}\int_{\Omega}\nabla v\nabla\phi\ d\Omega dt+\int_{0}^{T}\int_{\Omega}v\phi\ d\Omega dt\\
	&+\int_{0}^{T}\int_{\Omega}\nabla u\nabla \psi\ d\Omega dt-\int_{0}^{T}\int_{\Omega}v\psi\ d\Omega dt
	\end{split}
\end{equation}  and
\begin{equation}
	F(\phi,\psi)=\int_{0}^{T}\int_{\Omega}f\phi\ d\Omega dt.
\end{equation}

The unique solvability of the variational formulation \eqref{3pro1bf} can be investigated by the Banach-Ne\u{c}as-Babu\u{s}ka theorem which is stated below.
  \begin{theorem}\cite{Ern2004}\label{3BNB}
 	Let $Z$ be a Banach space and $U$ be a reflexive Banach space. Let $a:Z\times U\rightarrow\mathbb{R}$ be a bounded bilinear form and $f\in U^*$. Then there exists a unique $u\in Z$ satisfying $a(u,v)=f(v) \ \forall v\in U$ if and only if the following two conditions hold:
 	\begin{enumerate}[BNB1] 
 		\item (inf-sup stability): There exists $\beta>0$ such that, $\forall v\in Z,\quad\displaystyle \sup_{w\in U}\frac{a(v,w)}{\vert\vert w\vert\vert_U}\geq \beta\vert\vert v\vert\vert_Z$.\\
 		\item (injectivity): For $w\in U$, $(\forall v\in Z,\  a(v,w)=0)\Rightarrow (w=0).$ 
 	\end{enumerate}
    Moreover, the solution $u$ satisfies $\vert\vert u\vert\vert_Z\leq \frac{1}{\beta}\vert\vert f\vert\vert_{U^*}$.
 \end{theorem} 
\par  To assert the inf-sup stability condition, we will exploit an important inequality given in the following lemma:
\begin{lemma}\cite{inequref}
	For $u\in V$, the following inequality holds\begin{equation}\label{3ineq1}
		\frac{\vert\vert u\vert\vert_W^2}{\vert\vert u\vert\vert_V^2}\geq\frac{1}{6}.
	\end{equation}
\end{lemma}
Now in the following lemma, we prove the inf-sup stability condition for the bilinear form \eqref{3bili}.
\begin{lemma}\label{3continfsup}
	For any $(u,v)\in\mathcal{V}$, we have
	\begin{equation*}
		\displaystyle \sup_{0\neq (\phi,\psi)\in\mathcal{W}} \frac{a(u,v;\phi,\psi)}{\vert\vert(\phi,\psi)\vert\vert_\mathcal{W}}\geq C_s\vert\vert(u,v)\vert\vert_{\mathcal{V}},
	\end{equation*}
where $C_s$ is a constant that does not depend on $(u,v)$.
\end{lemma}
\begin{proof}
For $u\in V$, we define $w_u$ to be the unique solution of the problem 
\begin{equation*}
	\int_{0}^{T}\int_{\Omega}\partial_tu\phi\ d\Omega dt=\int_{0}^{T}\int_{\Omega}\nabla w_u\nabla \phi\ d\Omega dt\quad \forall\phi\in W.
\end{equation*}
As proven in \cite{stfempara}, we have
\begin{equation*}
	\vert\vert w_u\vert\vert_W=\vert\vert\partial_t u\vert\vert_{L^2(H^{-1}(\Omega))}.
\end{equation*}
So, we can write
\begin{equation*}
    \nv{u}^2=\nw{w_u}^2+\nw{u}^2.
\end{equation*}
	Let $\alpha>0$ (which we will fix later). We take
	\begin{equation*}
		\phi=6\alpha u+w_u+v
	\end{equation*}
	and
	\begin{equation*}
		\psi=6\alpha u+w_u-6\alpha v.
	\end{equation*}
Clearly, $\phi,\psi\in W$. \\
Now, 
\begin{equation*}
	\begin{split}
		\vert\vert\phi\vert\vert_W^2&\leq2((6\alpha\vert\vert u\vert\vert_W+\vert\vert w_u\vert\vert_W)^2+\vert\vert v\vert\vert_W^2)\\
		&\leq 4(36\alpha^2\vert\vert u\vert\vert_W^2+\vert\vert w_u\vert\vert_W^2)+2\vert\vert v\vert\vert_W^2\\
		&\leq\text{max}\{144\alpha^2,4,2\}(\vert\vert u\vert\vert_W^2+\vert\vert w_u\vert\vert_W^2+\vert\vert v\vert\vert_W^2).
	\end{split}
\end{equation*}
Similarly, we have
\begin{equation*}
	\vert\vert \psi\vert\vert_W^2\leq \text{max}\{144\alpha^2,4,72\alpha^2\}(\vert\vert u\vert\vert_W^2+\vert\vert w_u\vert\vert_W^2+\vert\vert v\vert\vert_W^2).
\end{equation*}
Then 
\begin{equation}\label{3lemma2ineq1}
	\vert\vert(\phi,\psi)\vert\vert_{\mathcal{W}}\leq C_1\vert\vert(u,v)\vert\vert_{\mathcal{V}},
\end{equation}
where $C_1$ is a constant depending on $\alpha$.\\
Now,
\begin{equation*}
	\begin{split}
		a(u,v;\phi,\psi)&
		=\int_{0}^{T}\int_{\Omega}\partial_tu(6\alpha u+w_u+v)\ d\Omega dt\\
		&+\int_{0}^{T}\int_{\Omega}\nabla v\nabla(6\alpha u+w_u+v)\ d\Omega dt\\
		&+\int_{0}^{T}\int_{\Omega}v(6\alpha u+w_u+v)\ d\Omega dt\\
		&+\int_{0}^{T}\int_{\Omega}\nabla u\nabla (6\alpha u+w_u-6\alpha v)\ d\Omega dt\\
		&-\int_{0}^{T}\int_{\Omega}v(6\alpha u+w_u-6\alpha v)\ d\Omega dt\\
		&\geq 3\alpha\vert\vert u(T)\vert\vert^2+\nw{w_u}^2+\int_{0}^{T}\int_{\Omega}\nabla w_u\nabla v\ d\Omega dt\\
		&+\int_{0}^{T}\int_{\Omega}\nabla v\nabla w_u\ d\Omega dt+\vert\vert v\vert\vert_W^2+6\alpha\vert\vert u\vert\vert_W^2+\frac{1}{2}\vert\vert u(T)\vert\vert^2.
 \end{split}
        \end{equation*}
        Using the inequality \eqref{3ineq1}, we have
        \begin{equation*}
            \begin{split}
       a(u,v;\phi,\psi) &\geq \vert\vert w_u\vert\vert_W^2+2\int_{0}^{T}\int_{\Omega}\nabla w_u\nabla v\ d\Omega dt+ \vert\vert v\vert\vert_W^2+\alpha\vert\vert u\vert\vert _W^2+\alpha\vert\vert w_u\vert\vert_W^2\\
		&\geq \alpha\vert\vert u\vert\vert_W^2+(1+\alpha)\vert\vert w_u\vert\vert_W^2+\vert\vert v\vert\vert_W^2-\frac{1}{2}\vert\vert v\vert\vert_W^2-2\vert\vert w_u\vert\vert_W^2.      
            \end{split}
        \end{equation*}
\
Choosing $\alpha=2$, we obtain
\begin{equation*}
	\begin{split}
		a(u,v;\phi,\psi)&\geq2\vert\vert u\vert\vert_W^2+\nw{w_u}^2+\frac{1}{2}\vert\vert v\vert\vert_W^2\\
		&\geq\frac{1}{2}\vert\vert(u,v)\vert\vert_{\mathcal{V}}^2.
	\end{split}
\end{equation*}
From \eqref{3lemma2ineq1}, we get
\begin{equation*}
	a(u,v;\phi,\psi)\geq C_s\vert\vert (\phi,\psi)\vert\vert_{\mathcal{W}}\vert\vert(u,v)\vert\vert_{\mathcal{V}},
\end{equation*}
where $C_s$ is a constant.\\
Then
\begin{equation*}
	\displaystyle \sup_{0\neq (\phi,\psi)\in\mathcal{W}} \frac{a(u,v;\phi,\psi)}{\vert\vert(\phi,\psi)\vert\vert_\mathcal{W}}\geq C_s\vert\vert(u,v)\vert\vert_{\mathcal{V}},
\end{equation*}
which completes the proof.
\end{proof}
We now prove the injectivity condition which is required in the BNB Theorem \ref{3BNB}.
\begin{lemma}
Let $(\phi,\psi)\in\mathcal{W}$ be such that $a(u,v;\phi,\psi)=0\quad\forall (u,v)\in\mathcal{V}$. Then $\phi=\psi=0$.
\end{lemma}
\begin{proof}
Let $\phi,\ \psi\in W$. For $\x\in\Omega$ and $t\in(0,T)$, define
\begin{equation*}
	u_{\psi}(\x,t)=\int_{0}^{t}\psi(\x,s)\ ds.
\end{equation*}
Then by definition, we have $u_{\psi}\in V$ \cite{Langer2021multi1}.\\
Suppose that
\begin{equation*}
   a(u,v;\phi,\psi)=0\quad\forall(u,v)\in\mathcal{V}. 
\end{equation*}
By taking $u=u_{\psi}$ and $v=\phi$ in above, we have
\begin{equation*}
	a(u,v;\phi,\psi)=\intq\phi\psi\dx+\nw{\phi}^2+\vert\vert\phi\vert\vert_{L^2(L^2(\Omega))}^2+\frac{1}{2}\vert\vert\nabla u_{\psi}(\x,T)\vert\vert^2-\intq \phi\psi\dx.
\end{equation*}
Since $a(u,v;\phi,\psi)=0$, we get
\begin{equation*}
	\nw{\phi}^2+\vert\vert\phi\vert\vert_{L^2(L^2(\Omega))}+\frac{1}{2}\vert\vert\nabla u_{\psi}(\x,T)\vert\vert^2=0.
\end{equation*}
Then
\begin{equation*}
	\nw{\phi}^2=0.
\end{equation*}
Hence
\begin{equation*}
	\phi=0.
\end{equation*}
In particular, we can write $a(u,v;\phi,\psi)=a(u,v;0,\psi)=0\quad\forall (u,v)\in\mathcal{V}$.\\\\
Now, we choose $u=u_{\psi}$ and $v=-\psi$ in above equation to get\\
\begin{equation*}
	0=a(u,v;0,\psi)=\intq\nabla u_{\psi}\nabla \psi\dx+\intq\psi^2\dx.
\end{equation*}
Then 
\begin{equation*}
	\frac{1}{2}\vert\vert\nabla u_{\psi}(\x,T)\vert\vert^2+\vert\vert\psi\vert\vert_{L^2(L^2(\Omega))}^2=0.
\end{equation*}
Thus, $\psi=0$ which completes the proof.
\end{proof}
As a consequences of the BNB Theorem \ref{3BNB}, we can guarantee the unique solvability of \eqref{3pro1bf} and it is stated in the following theorem.
\begin{theorem}
For a given $f\in W^*$, there exists a unique $(u,v)\in\mathcal{V}$ satisfying
\begin{equation*}
	a(u,v;\phi,\psi)=F(\phi,\psi)\quad\forall(\phi,\psi)\in\mathcal{W}
\end{equation*}
and \begin{equation*}
    \nv{(u,v)}\leq \frac{1}{C_s}\vert\vert f\vert\vert_{W^*}.
\end{equation*}
\end{theorem}
\begin{remark}
The initial condition can be treated as a Dirichlet condition on the space-time cylinder $\Omega\times(0,T)$. For non-homogeneous initial condition $u(\x,0)=u_0(\x)$, we can obtain the homogeneous variational formulation as: Find $(\overline{u},v)\in\mathcal{V}$ such that,
\begin{equation*}
	a(\overline{u},v;\phi,\psi)=F(\phi,\psi)-\int_{0}^{T}\int_{\Omega}\partial_t\overline{u}_0\phi\ d\Omega dt-\int_{0}^{T}\int_{\Omega}\nabla\overline{u}_0\nabla\psi\ d\Omega dt\quad \forall (\phi,\psi)\in\mathcal{W},
\end{equation*}
where we take $u=\overline{u}+\overline{u}_0$ and $\overline{u}_0\in\{w\in L^2(H_0^1(\Omega)):\partial_tw\in L^2(H^{-1}(\Omega))\}$ is an extension of $u_0\in H_0^1(\Omega)$.
\end{remark}
\section{Discrete formulation and error analysis}\label{3sec4dis}
This section starts with a brief introduction about B-Spline basis functions and the isogeometric space that is needed for the discrete analysis. 
\subsection{B-spline spaces}
For two positive integers $l$ and $p$, a set $\Xi:=\{0=\xi_1,\cdots,\xi_{l+p+1}=1\}$  with $\xi_1\leq\cdots\leq\xi_{l+p+1}$ is called a knot vector in [0,1]. In our analysis, open knot vectors are considered, i.e. we assume that $\xi_1=\cdots=\xi_{p+1}=0$ and $\xi_l=\cdots=\xi_{l+p+1}=1$. Using the Cox-De Boor formula \cite{nurbsbook}, the B-spline basis functions can be defined recursively. We first define the B-splines of zero degree as follows:
\begin{equation*}
	\wh{b}_{j,0}(\zeta)=\begin{cases}
		1 \mbox{ if }\xi_j\leq \zeta<\xi_{j+1},\\
		0 \mbox{ otherwise. }
	\end{cases}
\end{equation*}
Then the higher degree B-splines are defined as:
\begin{equation*}
	\wh{b}_{j,p}(\zeta)=\frac{\zeta-\xi_j}{\xi_{j+p}-\xi_{j}}\wh{b}_{j,p-1}(\zeta)+\frac{\xi_{j+p+1}-\zeta}{\xi_{j+p+1}-\xi_{j+1}}\wh{b}_{j+1,p-1}(\zeta).
\end{equation*}
Here, the division by $0$ is taken to be $0$. Let $h$ denote the mesh size, i.e. $h:=\mbox{max}\{|\xi_{j+1}-\xi_{j}|:1\leq j\leq l\}$. Hence, the univariate spline space can be defined as
\begin{equation*}
	\wh{\mathcal{S}}_{h,p}:=\mbox{span}\{\wh{b}_{j,p}:1\leq j\leq l\}.
\end{equation*}
 We take B-spline functions that depend on both time and space where the spatial domain is of dimension $d$. Given integers $l_r,p_r$ for $1\leq r\leq d$ and $l_t,p_t$, we consider $d+1$ knot vectors $\Xi_r=\{\xi_{r,1},\cdots,\xi_{r,l_r+p_r+1}\}$ and $\Xi_t=\{\xi_{t,1},\cdots,\xi_{t,l_t+p_t+1}\}$. Let $h_r$ be the mesh size of the knot vector $\Xi_r$ for $r=1,\cdots,d$. Then we take $h_s=\mbox{max}\{h_r:1\leq r\leq d\}$ and by $h_t$ we denote the mesh size in time direction. The following quasi-uniformity assumption is satisfied by the knot vectors.\\\\
\textbf{Assumption 1.} If span $(\xi_{r,i},\xi_{r,i+1})$ is a non-empty knot  of $\Xi_r$, then we have $\beta h_s\leq \xi_{r,i+1}-\xi_{r,i}\leq h_s$ for $r=1,\cdots,d$ and if $(\xi_{t,i},\xi_{t,i+1})$ is a non-empty knot span of $\Xi_t$, then we have $\beta h_t\leq \xi_{t,i+1}-\xi_{t,i}\leq h_t$, where $0<\beta\leq1$ is independent of $h_s$ and $h_t$.\\\\
We define the multivariate B-splines on the parametric domain $\wh{\Omega}:=(0,1)^d$ by taking the tensor product of the univariate B-splines as follows:
\begin{equation*}
	\wh{B}_{\mathbf{i_s},\mathbf{p_s}}(\mathbf{\zeta}):=\wh{b}_{i_1,p_1}(\zeta_1)\cdots\wh{b}_{i_d,p_d}(\zeta_d),
\end{equation*}
where $\mathbf{\zeta}=(\zeta_1,\cdots,\zeta_d)$, $\mathbf{i_s}=(i_1,\cdots,i_d)$, $\mathbf{p_s}=(p_1,\cdots,p_d)$. In our analysis, we fix $p_1=\cdots=p_d=p_s$. The B-spline basis function on $\wh{\Omega}\times(0,1)$ is defined by
\begin{equation*}
	\wh{B}_{\mathbf{i},\mathbf{p}}(\mathbf{\zeta},\tau):=\wh{B}_{\mathbf{i_s},\mathbf{p_s}}(\zeta)\wh{b}_{i_t,p_t}(\tau),
\end{equation*}
where $\mathbf{i}=(\mathbf{i_s},i_t)$ and $\mathbf{p}=(\mathbf{p_s},p_t)$.
Then the multivariate spline space is given as
\begin{equation*}
	\wh{\mathcal{S}}_{h,\mathbf{p}}:=\mbox{span}\{	\wh{B}_{\mathbf{i},\mathbf{p}}:1\leq i_r\leq l_r, \mbox{ for } r=1,\cdots,d \mbox{ and }1\leq i_t\leq l_t\},
\end{equation*}
where $h=\mbox{max}\{h_s,h_t\}$. Since the spline spaces have tensor product structure, we have that
\begin{equation*}
	\wh{\mathcal{S}}_{h,\mathbf{p}}=\wh{\mathcal{S}}_{h_s,\mathbf{p_s}}\otimes\wh{\mathcal{S}}_{h_t,p_t},
\end{equation*}
where $\wh{\mathcal{S}}_{h_s,\mathbf{p_s}}=\mbox{span}\{\wh{B}_{\mathbf{i_s},\mathbf{p_s}}:1\leq i_r\leq l_r \mbox{ for } r=1,\cdots,d\}$.\\\\
\textbf{Assumption 2:} Here, $p_s,p_t\geq 1$ so that $\wh{\mathcal{S}}_{h_s,\mathbf{p_s}}\subset C^0(\wh{\Omega})$ and $\wh{\mathcal{S}}_{h_t,{p_t}}\subset C^0((0,1))$.\\
\par Let $\mathbf{G}:\wh{\Omega}\rightarrow\Omega$ be a mapping satisfying the assumption that $\mathbf{G}^{-1}$ has derivatives of any order and they are  piecewise bounded. Let $\mathbf{G}(\mathbf{\zeta})=\x$, where $\x=(x_1,\cdots,x_d)$ and let $t:=T\tau$. The space-time cylinder is parameterized using the function $\tilde{\mathbf{G}}:\wh{\Omega}\times(0,1)\rightarrow\Omega\times(0,T)$, such that $\tilde{\mathbf{G}}(\mathbf{\zeta},\tau):=(\mathbf{G}(\mathbf{\zeta}),T\tau)=(\x,t)$.\\
To impose the initial and boundary conditions, we consider 
\begin{equation*}
	\wh{V}_h=\{\wh{\phi}_h\in\wh{\mathcal{S}}_{h,\mathbf{p}}:\wh{\phi}_h=0\mbox{ on }\partial\wh{\Omega}\times(0,1)\mbox{ and }\wh{\phi}_h=0\mbox{ in }\wh{\Omega}\times{0}\}.
\end{equation*}
Let \begin{equation*}
\wh{V}_{h_s}=\{\wh{\phi}\in\wh{\mathcal{S}}_{h_s,\mathbf{p_s}}:\wh{\phi}=0\mbox{ on }\partial\wh{\Omega}\}=\mbox{span}\{\wh{b}_{j_1,p_s}\cdots\wh{b}_{j_d,p_s}:2\leq j_r\leq l_r-1;r=1,\cdots,d\}
\end{equation*}
and
\begin{equation*}
	\wh{V}_{h_t}=\{\wh{\phi}\in\wh{\mathcal{S}}_{h_t,p_t}:\wh{\phi}(0)=0\}=\mbox{span}\{\wh{b}_{j_t,p_t}:2\leq j_t\leq l_t\}.
\end{equation*}
Then we can write $\wh{V}_h=\wh{V}_{h_s}\otimes\wh{V}_{h_t}$.\\
By considering the colexicographical reordering of the degrees-of-freedom, we have
\begin{equation*}
	\wh{V}_{h_s}=\{\wh{B}_{j,\mathbf{p_s}}:j=1,\cdots,n_s\}
\end{equation*}
and
\begin{equation*}
	\wh{V}_{h_t}=\{\wh{b}_{j,p_t}:j=1,\cdots,n_t\},
\end{equation*}
where $n_s=\Pi_{r=1}^{d}(l_r-2)$ and $n_t=l_t-1$. Then 
\begin{equation*}
	\wh{V}_h=\mbox{span}\{\wh{B}_{j,\mathbf{p}}: j=1,\cdots, N\},
\end{equation*}
where, $N=n_sn_t$.\\
The isogeometric space can be defined using a push-forward of $\wh{V}_h$ through the parametarization $\tilde{\mathbf{F}}$ as 
\begin{equation}\label{3isvh}
	V_h:=\{B_{j,\mathbf{p}}=\wh{B}_{j,\mathbf{p}}\circ\tilde{\mathbf{F}}^{-1}:j=1,\cdots,N\}.
\end{equation}
We also have $V_h=V_{h_s}\otimes V_{h_t}$, where
\begin{equation*}
	V_{h_s}:=\{B_{j,\mathbf{p_s}}=\wh{B}_{j,\mathbf{p_s}}\circ\mathbf{F}^{-1}:j=1,\cdots,n_s\}
\end{equation*}
and
\begin{equation*}
	V_{h_t}:=\{b_{j,p_t}=\wh{b}_{j,p_t}(\cdot/T):j=1,\cdots,n_t\}.
\end{equation*}\\
\par To introduce the discrete space-time variational formulation, we consider the finite dimensional subspace $V_h\subset V$ given in \eqref{3isvh}, and set $W_h=V_h$. Consider the product spaces  $\mathcal{V}_h=V_h\times W_h$ and $\mathcal{W}_h=W_h\times W_h$. Then the isogeometric discretization of \eqref{3pro1st} reads : Find $(u_h,v_h)\in\mathcal{V}_h$ such that
\begin{equation}\label{3pro1dstf}
	\begin{split}
		\int_{0}^{T}\int_{\Omega}\partial_tu_h\phi_h\ d\Omega dt&+\int_{0}^{T}\int_{\Omega}\nabla v_h\nabla\phi_h\ d\Omega dt+\int_{0}^{T}\int_{\Omega}v_h\phi_h\ d\Omega dt\\
		&=\int_{0}^{T}\int_{\Omega}f\phi_h\ d\Omega dt\quad\forall \phi_h\in W_h,\\
		\int_{0}^{T}\int_{\Omega}\nabla u_h\nabla \psi_h\ d\Omega dt&=\int_{0}^{T}\int_{\Omega}v_h\psi_h\ d\Omega dt\quad\forall \psi_h\in W_h.
	\end{split}
\end{equation}
We observe that problem \eqref{3pro1dstf} is equivalent to the following variational problem: Find $(u_h,v_h)\in\mathcal{V}_h$ such that
\begin{equation}\label{3pro1dst}
	a(u_h,v_h;\phi_h,\psi_h)=F(\phi_h,\psi_h)\quad\forall (\phi_h,\psi_h)\in\mathcal{W}_h.
\end{equation}
Since $\mathcal{V}_h\subset\mathcal{V}$ and $\mathcal{W}_h\subset\mathcal{W}$, from \eqref{3pro1dst} and \eqref{3pro1bf} we get the following Galerkin orthogonality
\begin{equation}\label{3galor}
	a(u-u_h,v-v_h;\phi_h,\psi_h)=0\quad\forall (\phi_h,\psi_h)\in\mathcal{W}_h,
\end{equation} which plays an important role in proving the error estimates.
\par For $u\in V$, we define $w_{h,u}\in W_h$ as the solution of the discrete problem 
\begin{equation*}
	\intq\partial_t u\phi_h\dx=\intq\nabla w_{h,u}\nabla \phi_h\dx\quad\forall\phi_h\in W_h.
\end{equation*}
As given in \cite{stfempara}, we consider the following discrete norm 
\begin{equation}\label{3disnorm}
\nvh{u}^2=\nw{w_{h,u}}^2+\nw{u}^2
\end{equation}
and we can define the discrete norm on product space as
\begin{equation*}
    \vert\vert(u_h,v_h)\vert\vert_{\mathcal{V}_h}^2=\nvh{u_h}^2+\nw{v_h}^2.
\end{equation*}
We also have $\nvh{u}^2\leq\nv{u}^2$, and since $u_h\in V$ from \eqref{3ineq1} and \eqref{3disnorm} we have the following inequality
\begin{equation*}
	\frac{\nw{u_h}^2}{\nvh{u_h}^2}\geq\frac{1}{6}.
\end{equation*}
Similar to the continuous case, we prove the discerete inf-sup stability condition in the following lemma.
\begin{lemma}
For $(u_h,v_h)\in\mathcal{V}_h$, we have
\begin{equation}\label{3disinfsupin}
		\displaystyle \sup_{0\neq (\phi_h,\psi_h)\in\mathcal{W}_h} \frac{a(u_h,v_h;\phi_h,\psi_h)}{\vert\vert(\phi_h,\psi_h)\vert\vert_{\mathcal{W}_h}}\geq C_s\vert\vert(u_h,v_h)\vert\vert_{\mathcal{V}_h}.
\end{equation}
\end{lemma}
\begin{proof}
The inequality \eqref{3disinfsupin} is proved following the same strategies used in the Lemma \ref{3continfsup}. We only outline the main steps here. Let $(u_h,v_h)\in\mathcal{V}_h$. We take $\alpha=2$ and consider
	\begin{equation*}
	\phi_h=6\alpha u_h+w_{h,u_h}+v_h
\end{equation*}
and
\begin{equation*}
	\psi_h=6\alpha u_h+w_{h,u_h}-6\alpha v_h.
\end{equation*}
Then we have
\begin{equation*}
\vert\vert(\phi_h,\psi_h)\vert\vert_{\mathcal{W}_h}\leq C_1\vert\vert(u_h,v_h)\vert\vert_{\mathcal{V}_h}.
\end{equation*}
As in the proof of Lemma \ref{3continfsup}, we get
\begin{equation*}
	a(u_h,v_h;\phi_h,\psi_h)\geq  \frac{1}{2}\vert\vert(u_h,v_h)\vert\vert_{\mathcal{V}_h}^2.
\end{equation*}
The inequality \eqref{3disinfsupin} follows as in the continuous case.
\end{proof}
The unique solvability of the space-time isogeometric scheme \eqref{3pro1dst} follows from the discrete inf-sup stability condition. 
\subsection{Discrete system}
The discrete problem \eqref{3pro1dstf} is equivalent to the matrix system  given by
\begin{equation}\label{3linsys}
	\begin{split}
		\mathbf{Wu}+ \mathbf{Kv}+\mathbf{Mv}&=\mathbf{f},\\
		\mathbf{Ku}&=\mathbf{Mv},
	\end{split}
\end{equation}
where
\begin{equation*}
	\begin{split}
		[\mathbf{W}]_{i,j}&=\int_{0}^{T}\int_{\Omega}\partial_tB_{j,\mathbf{p}}B_{i,\mathbf{p}}\ d\Omega dt,\\
		[\mathbf{K}]_{i,j}&=\int_{0}^{T}\int_{\Omega}\nabla B_{j,\mathbf{p}}\nabla B_{i,\mathbf{p}}\ d\Omega dt,\\
		[\mathbf{M}]_{i,j}&=\int_{0}^{T}\int_{\Omega}B_{j,\mathbf{p}}B_{i,\mathbf{p}}\ d\Omega dt\quad\mbox{for }i,j=1,\cdots,N
	\end{split}
\end{equation*}
and
\begin{equation*}
    [\mathbf{f}]_i=\int_{0}^{T}\int_{\Omega} fB_{i,\mathbf{p}}\ d\Omega dt\quad\mbox{for }i=1,\cdots,N.
\end{equation*}
Due to the 
tensor product structure of $V_h$, the above matrices can be written as a Kronecker product of matrices in space and time variables as follows \cite{LOLI20202586}:
\begin{equation*}
\mathbf{W}=W_t\otimes M_s,\quad \mathbf{K}=M_t\otimes K_s \quad\mbox{and}\quad\mathbf{M}=M_t\otimes M_s,
\end{equation*}
where 
\begin{equation*}
	[W_t]_{i,j}=\int_{0}^{T} b_{j,p_t}'b_{i,p_t}\ dt,\quad [M_t]_{i,j}=\int_{0}^{T}b_{j,p_t}b_{i,p_t}\ dt \mbox{ for }i,j=1,\cdots,n_t
\end{equation*}
and
\begin{equation*}
		[K_s]_{i,j}=\int_{\Omega}\nabla B_{j,\mathbf{p_s}}\nabla B_{i,\mathbf{p_s}}\ d\Omega,\quad
		[M_s]_{i,j}=\int_{\Omega} B_{j,\mathbf{p_s}} B_{i,\mathbf{p_s}}\ d\Omega\mbox{ for }i,j=1,\cdots,n_s.
\end{equation*}

\subsection{Error analysis}
In order to derive the a priori error estimates, we first need to prove the following Cea's-like lemma which is a consequence of the Galerkin orthogonality \eqref{3galor} and the discrete inf-sup stability condition \eqref{3disinfsupin}.
\begin{lemma}\label{3pro1cea}
	Let $(u,v)\in\mathcal{V}$ and $(u_h,v_h)\in\mathcal{V}_h$ be the solution of \eqref{3pro1bf} and \eqref{3pro1dst} respectively. Then
	\begin{equation*}
		\displaystyle\vert\vert(u-u_h,v-v_h)\vert\vert_{\mathcal{V}_h}\leq C\inf_{(z_h,w_h)\in\mathcal{V}_h}\vert\vert(u-z_h,v-w_h)\vert\vert_{\mathcal{V}}.
	\end{equation*}
\end{lemma}
\begin{proof}
Let $(z_h,w_h)\in\mathcal{V}_h$. Then by the discrete inf-sup stability condition \eqref{3disinfsupin}, we have
\begin{equation*}
	\begin{split}
		C\vert\vert(u_h-z_h,v_h-w_h)\vert\vert_{\mathcal{V}_h}&\leq\sup_{0\neq(\phi_h,\psi_h)\in\mathcal{W}_h} \frac{a(u_h-z_h,v_h-w_h;\phi_h,\psi_h)}{\vert\vert(\phi_h,\psi_h)\vert\vert_{\mathcal{W}_h}}\\
		&=\sup_{0\neq(\phi_h,\psi_h)\in\mathcal{W}_h} \frac{a(u_h-u,v_h-v;\phi_h,\psi_h)+a(u-z_h,v-w_h;\phi_h,\psi_h)}{\vert\vert(\phi_h,\psi_h)\vert\vert_{\mathcal{W}_h}}\\
		&\leq C'\vert\vert (u-z_h,v-w_h)\vert\vert_{\mathcal{V}}.
	\end{split}
\end{equation*}
Hence, by the triangular inequality, we get the desired result.
\end{proof}
Next, we recall the standard approximation result for the isogeometric spaces which is stated in the following theorem.
\begin{theorem}\cite{montardini2020space,LOLI20202586}\label{appres} Let $q_t\geq 1$, $q_s\geq 1$ with $u\in V\cap (H^{q_t}(0,T)\otimes H^1(\Omega))\cap( H^1(0,T)\otimes H^{q_s}(\Omega))$ and $v\in W\cap  (H^{q_t}(0,T)\otimes H^1(\Omega))\cap (H^1(0,T)\otimes H^{q_s}(\Omega))$. Then there exists a projection $\Pi_h: L^2(0,T)\otimes L^2(\Omega)\rightarrow V_h$ such that
	\begin{equation*}
		\vert\vert u-\Pi_h u\vert\vert_V\leq C(h_t^{m_t-1}\vert\vert u\vert\vert_{H^{m_t}(0,T)\otimes H^1(\Omega))}+h_s^{m_s-1}\vert\vert u\vert\vert_{H^1(0,T)\otimes H^{m_s}(\Omega)})
	\end{equation*}
    and 
    \begin{equation*}
		\vert\vert v-\Pi_h v\vert\vert_W\leq C(h_t^{m_t-1}\vert\vert u\vert\vert_{H^{m_t}(0,T)\otimes H^1(\Omega))}+h_s^{m_s-1}\vert\vert u\vert\vert_{H^1(0,T)\otimes H^{m_s}(\Omega)}),
	\end{equation*}
where $m_s:=\mbox{min}\{q_s,p_s+1\}$, $m_t:=\mbox{min}\{q_t,p_t+1\}$. 
\end{theorem}
We can now derive the following a priori error estimates for the space-time isogeometric scheme \eqref{3pro1dst}.
\begin{theorem}\label{3pro1error}
For two positive integers $q_t,\ q_s\geq 1,$  let $(u,v)\in\mathcal{V}$ be the solution of \eqref{3pro1bf} with $u\in V\cap (H^{q_t}(0,T)\otimes H^1(\Omega))\cap (H^1(0,T)\otimes H^{q_s}(\Omega))$,  $v\in W\cap  (H^{q_t}(0,T)\otimes H^1(\Omega))\cap (H^1(0,T)\otimes H^{q_s}(\Omega))$  and $(u_h,v_h)\in\mathcal{V}_h$ be the solution of \eqref{3pro1dst}. Then
\begin{equation*}
	\begin{split}
			\nvh{u-u_h}^2+\nw{v-v_h}^2\leq 
			&\ C[h^{2m_t-2}(\vert\vert u\vert\vert_{H^{m_t}(0,T)\otimes H^1(\Omega)}^2+\vert\vert v\vert\vert_{H^{m_t}(0,T)\otimes H^1(\Omega)}^2)\\
			&+h^{2m_s-2}(\vert\vert u\vert\vert_{H^1(0,T)\otimes H^{m_s}(\Omega)}^2+\vert\vert v\vert\vert_{H^1(0,T)\otimes H^{m_s}(\Omega)}^2)],
	\end{split}
\end{equation*}
where $m_s:=\mbox{min}\{q_s,p_s+1\}$, $m_t:=\mbox{min}\{q_t,p_t+1\}$. 
\end{theorem}
\begin{proof}
The proof follows by combining the approximation result given in Theorem \ref{appres} and Lemma \ref{3pro1cea}. 
\end{proof}
 \section{Numerical results}\label{3sec5numerics}
In this section, we will present several numerical experiments to validate the theoretical convergence result for the numerical scheme \eqref{3pro1dstf}. All the simulations are performed on a Intel Xeon processor having 256 GB RAM running at 3.90 GHz,  with Matlab R2024a and GeoPDEs toolbox \cite{geopde}. . In all experiments, we set $h_s=h_t=h$ (mesh size) and $p_s=p_t=p$ (degree of basis functions). We obtain the numerical solution of the system \eqref{3pro1} on meshes with $h=\frac{1}{4},\ \frac{1}{8},\ \frac{1}{16},\ \frac{1}{32},\ \frac{1}{64}$ and basis of degree $p=1,\ 2,\ 3,\ 4$. The problem domains $\Omega$ taken in the experiments are displayed in Figure \ref{3figdom}. The final time level is set to be 1 in all the examples. Note that norm $\vert\vert\cdot\vert\vert_{L^2(H_0^1(\Omega))\cap H^1(L^2(\Omega))}$ is an upper bound for the norm $\vert\vert\cdot\vert\vert_{V_h}$, which is easily computable. In each experiment, we compute the relative errors for $u_h$ in  ${L^2(H_0^1(\Omega))\cap H^1(L^2(\Omega))}$ norm and for $v_h$ in ${L^2(H_0^1(\Omega))}$ norm. We also have plotted the errors in $L^2(L^2(\Omega))$ norm,  even though the theoretical convergence estimates in this norm are not proven. We define \begin{equation*}
E_u^1=\frac{\vert\vert u-u_h\vert\vert_{L^2(H_0^1(\Omega))\cap H^1(L^2(\Omega))}}{\vert\vert u\vert\vert_{L^2(H_0^1(\Omega))\cap H^1(L^2(\Omega))}}\quad E_u^2=\frac{\vert\vert u-u_h\vert\vert_{L^2(L^2(\Omega))}}{\vert\vert u\vert\vert_{L^2(L^2(\Omega))}}
\end{equation*}
and 
\begin{equation*}
E_v^1=\frac{\vert\vert v-v_h\vert\vert_{L^2(H_0^1(\Omega))}}{\vert\vert u\vert\vert_{L^2(H_0^1(\Omega))}}\quad E_v^2=\frac{\vert\vert v-v_h\vert\vert_{L^2(L^2(\Omega))}}{\vert\vert v\vert\vert_{L^2(L^2(\Omega))}}.
\end{equation*}
\\\\
\begin{figure}[H]
	\centering
	\begin{subfigure}{0.3\textwidth}
		\includegraphics[height=5.5cm,width=5.5cm]{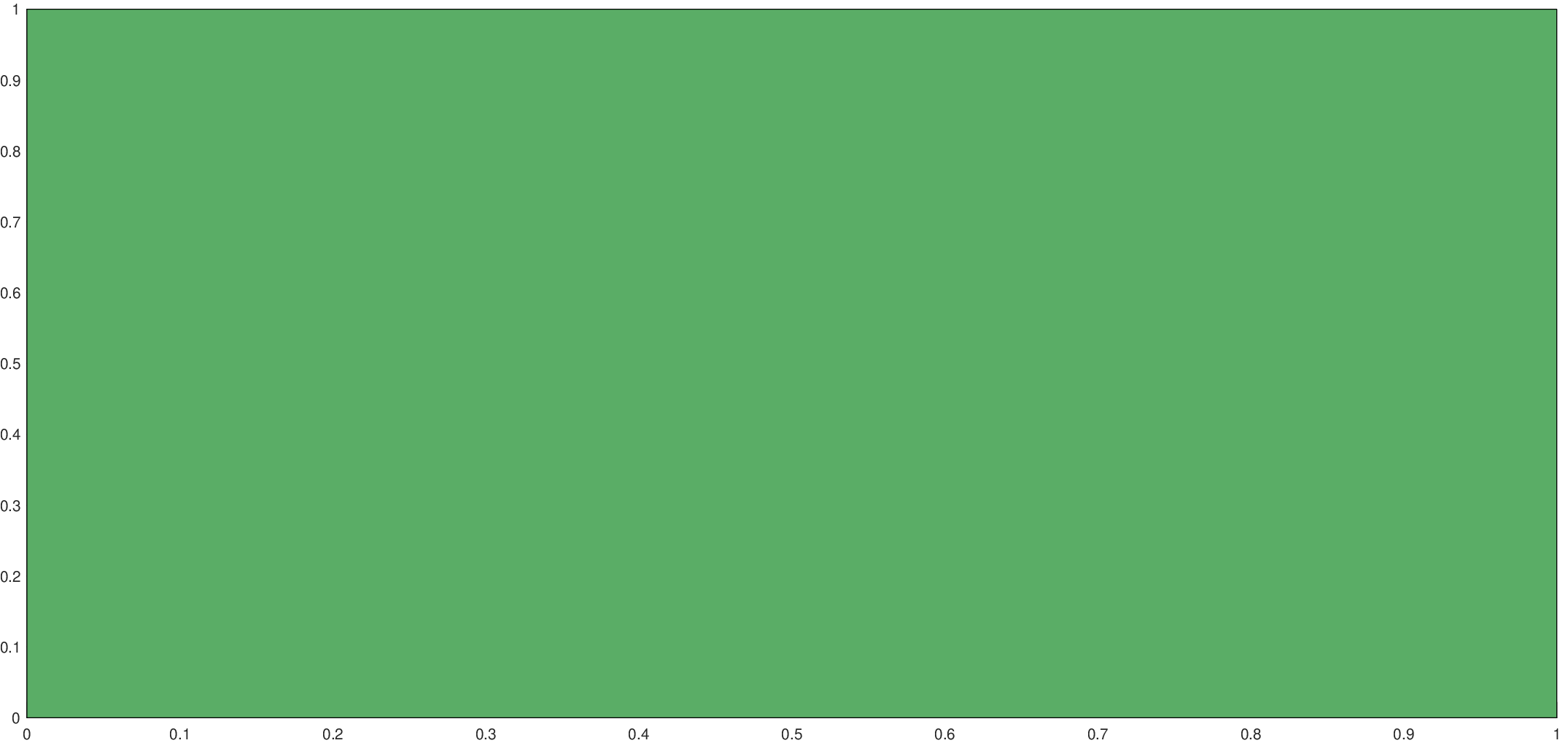}
		\caption{Square shaped domain.}
		\label{fig:square}
	\end{subfigure}
	\hspace{1.5cm}
	\begin{subfigure}{0.3\textwidth}
		\includegraphics[height=5.5cm,width=5.5cm]{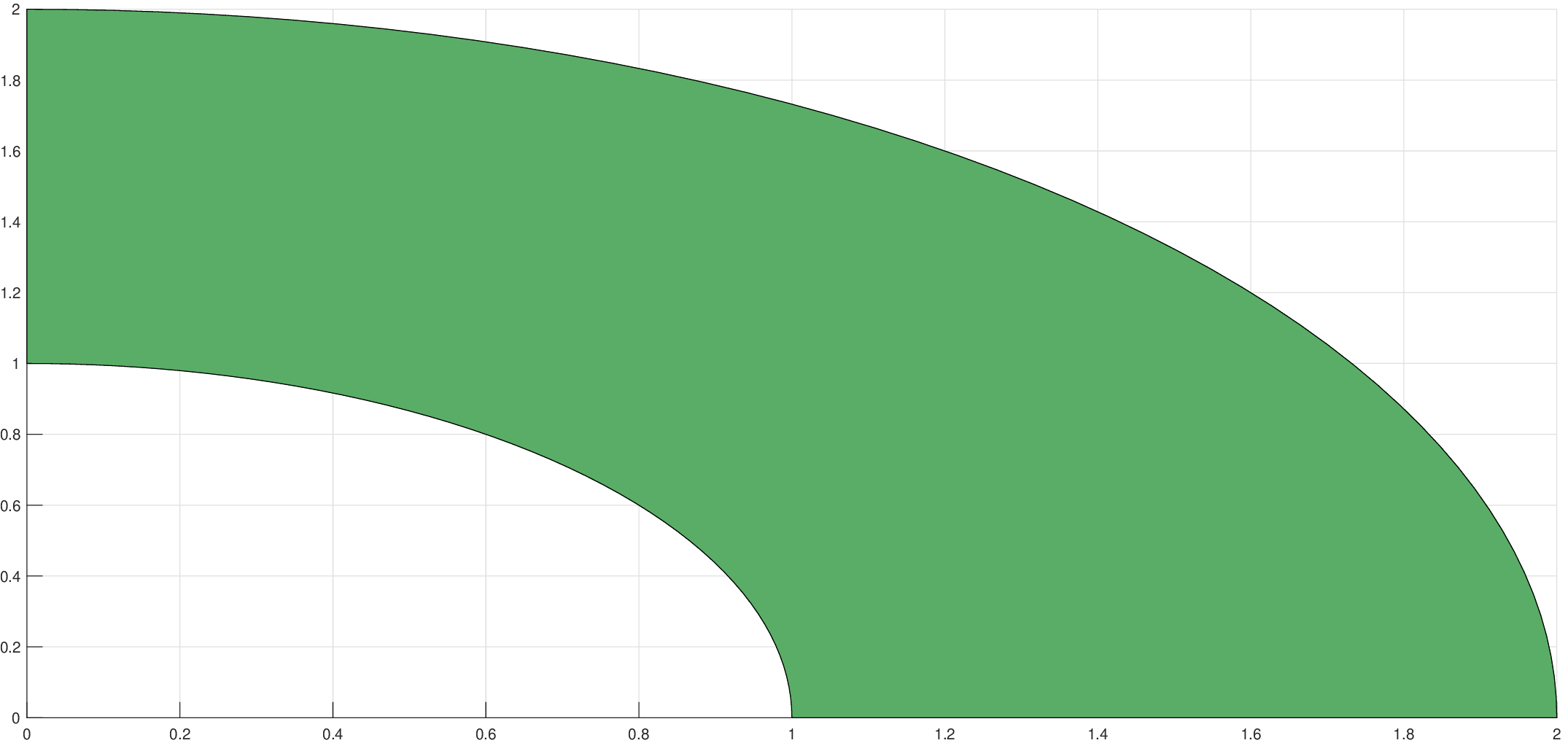}
		\caption{Ring shaped domain.}
		\label{fig:Ring shaped}
	\end{subfigure}
	\caption{Computational domains.}
	\label{3figdom}
\end{figure}
\hspace{-0.6cm}\textbf{Example 1:} In this example, we consider the domain $\Omega=(0,1)\times(0,1)$ displayed in Figure \ref{fig:square}. We choose the function $f$, the initial and boundary conditions in problem \eqref{3eqnmain} such that $u(x,y,t)=\sin(\pi t)\sin(\pi x)\sin(\pi y)$. In Figures \ref{fig:square_u} and \ref{fig:square_v}, the relative errors $E_u^1,\ E_u^2,\ E_v^1,\ E_v^2$ are displayed. From Figures \ref{fig:square_u_H1} and \ref{fig:square_v_H1}, we observe that the convergence rate in $L^2(H_0^1(\Omega))\cap H^1(L^2(\Omega))$ and $L^2(H_0^1(\Omega))$ norms is $\mathcal{O}(h^p)$ which confirms our theoretical estimate (Theorem \ref{3pro1error}). We also display the relative errors in $L^2(L^2(\Omega))$ norm in Figures \ref{fig:square_u_L2} and \ref{fig:square_v_L2} showing that the order of convergence in $L^2(L^2(\Omega))$ norm is $\mathcal{O}(h^{p+1})$.\\\\
\begin{figure}[H]
\begin{subfigure}{0.4\textwidth}
    \includegraphics[height=6cm,width=8cm]{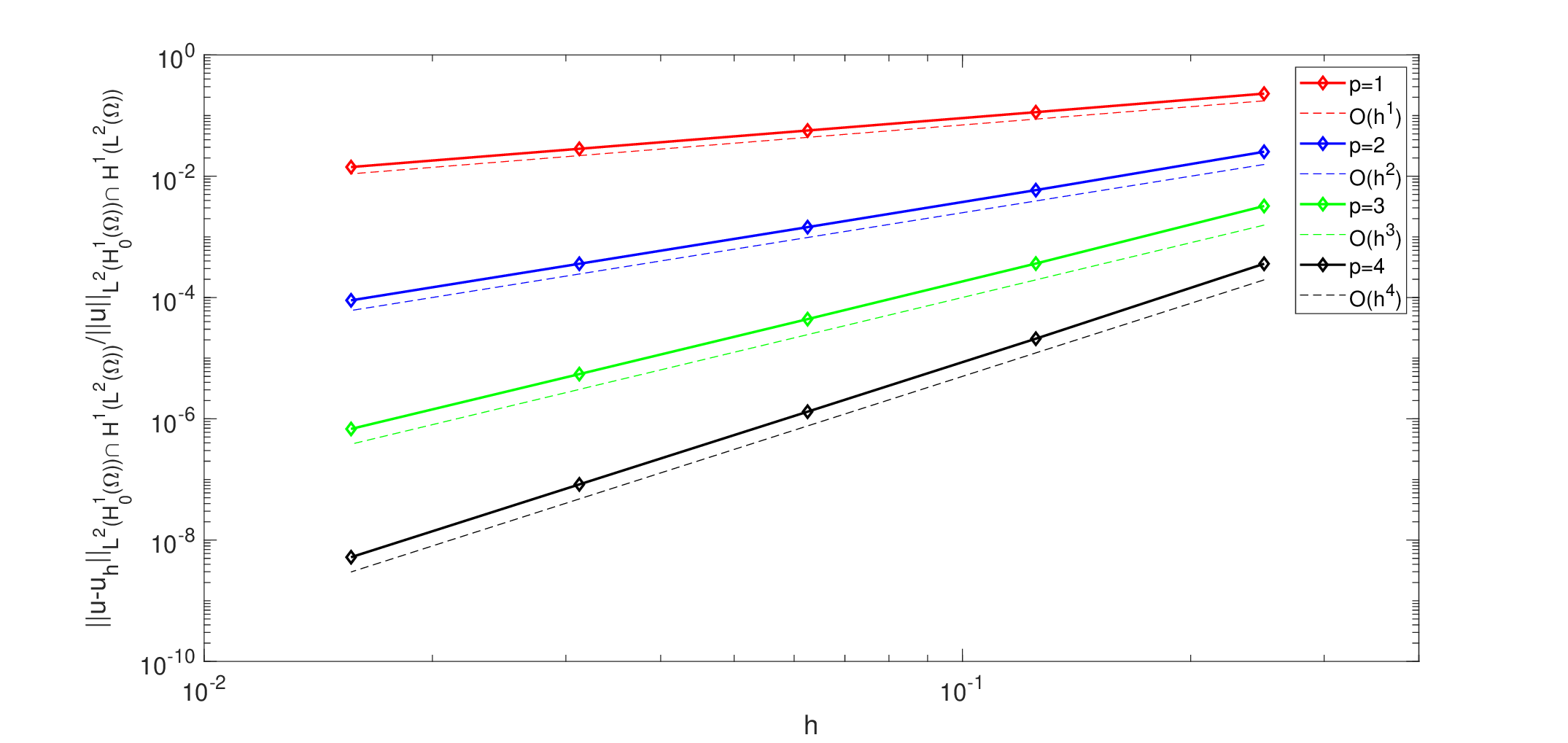}
    \caption[Caption]{}
    \label{fig:square_u_H1}
\end{subfigure}
\hspace{1cm}
\begin{subfigure}{0.4\textwidth}
    \includegraphics[height=6cm,width=8cm]{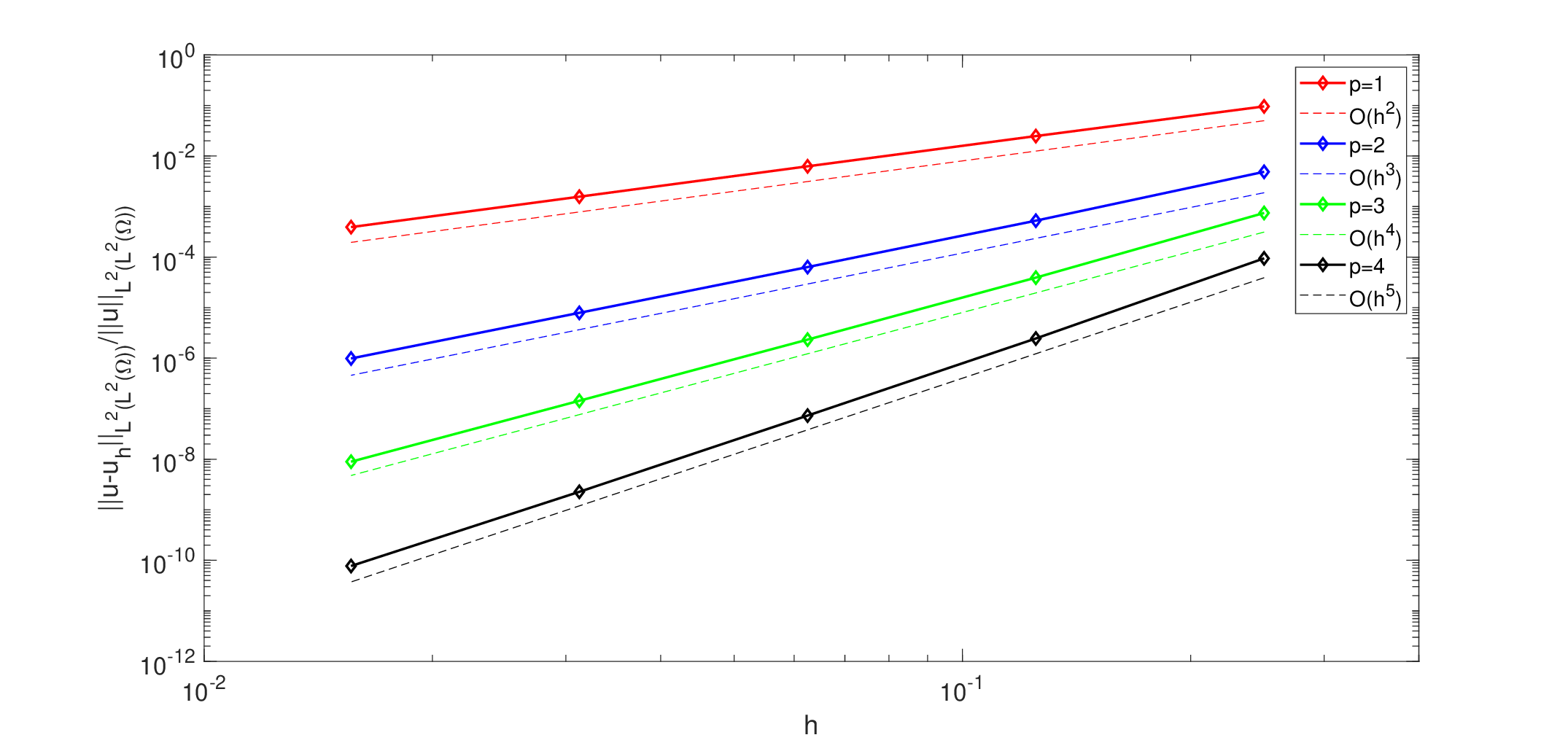}
     \caption[Caption]{}
    \label{fig:square_u_L2}
\end{subfigure}
\caption{Relative errors $E_u^1$ and $E_u^2$ for square domain.}
\label{fig:square_u}
\end{figure}

\begin{figure}[H]
\begin{subfigure}{0.4\textwidth}
    \includegraphics[height=6cm,width=8cm]{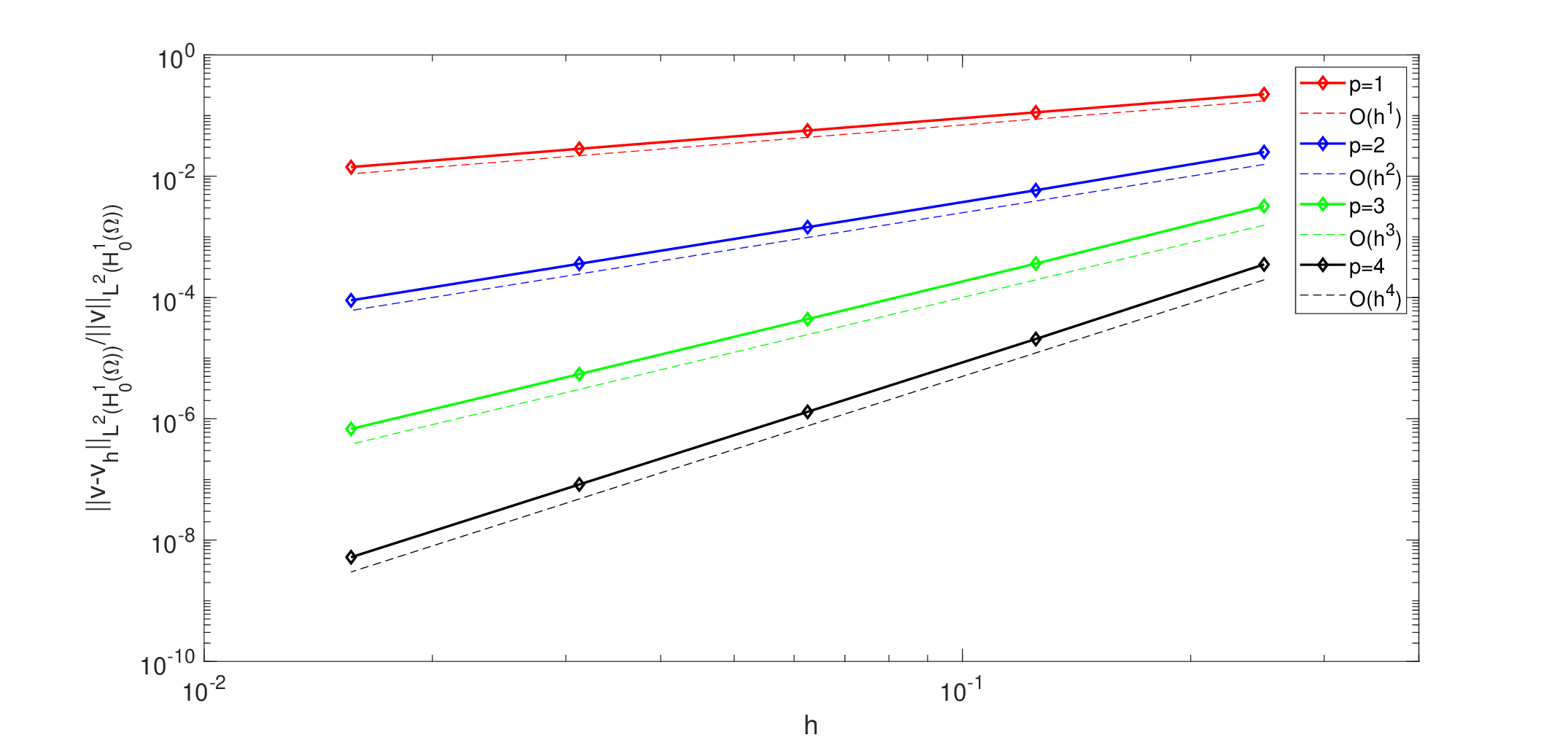}
     \caption[Caption]{}
    \label{fig:square_v_H1}
\end{subfigure}
\hspace{1cm}
\begin{subfigure}{0.4\textwidth}
    \includegraphics[height=6cm,width=8cm]{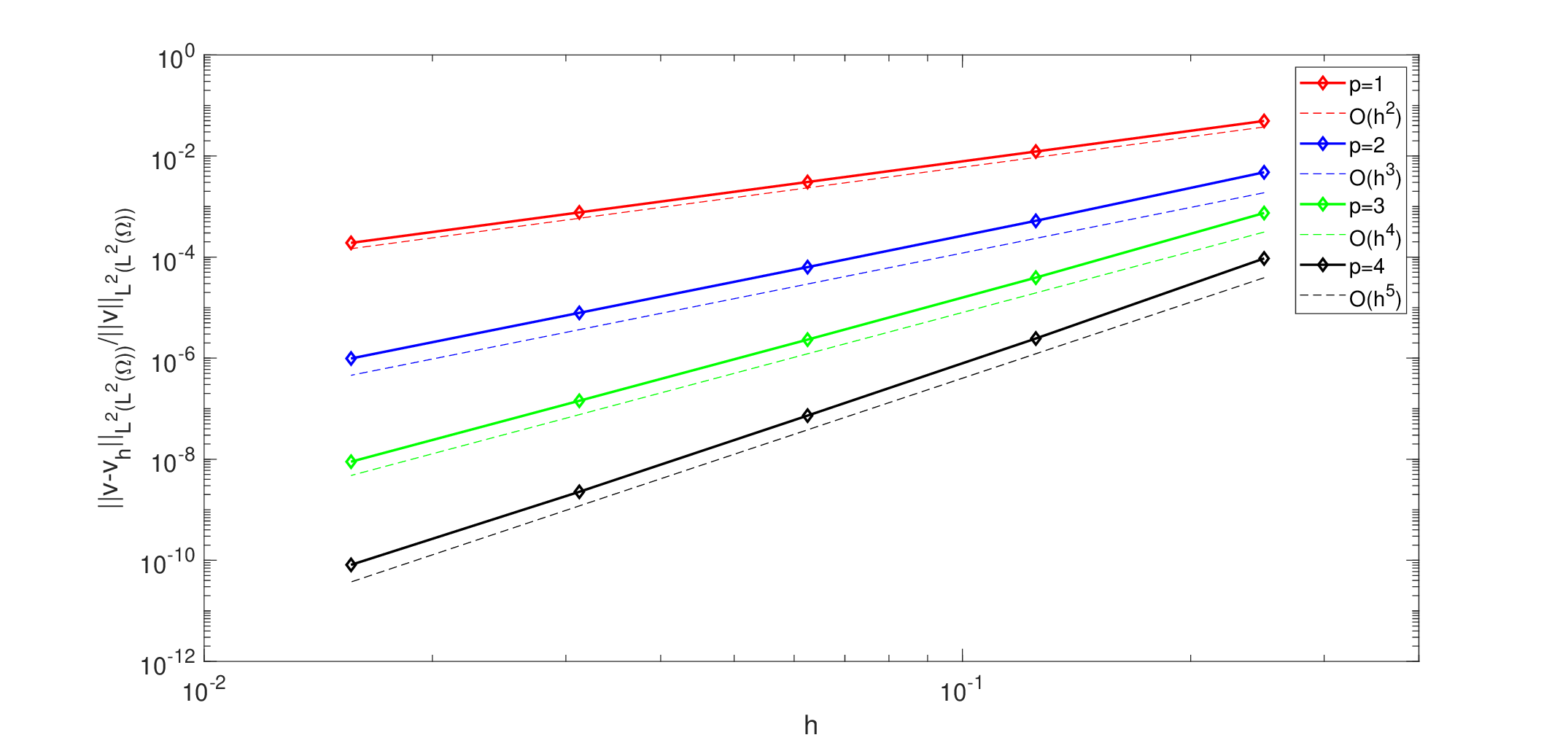}
     \caption[Caption]{}
    \label{fig:square_v_L2}
\end{subfigure}
\caption{Relative errors $E_v^1$ and $E_v^2$ for square domain.}
\label{fig:square_v}
\end{figure}

\hspace{-0.6cm}\textbf{Example 2:} In this example, we consider a non-convex ring shaped domain (see Figure \ref{fig:Ring shaped}). The initial and boundary conditions with the source term in \eqref{3eqnmain} are chosen such that the exact solution is $u(x,y,t)=tx^3y^3(x^2 + y^2 - 1)^3(x^2 + y^2 - 4)^3 $. We display the relative errors for $u_h$ and $v_h$ in Figure \ref{fig:ring_u} and Figure \ref{fig:ring_v} respectively. Though the theoretical results are proved for convex domains, we observe the optimal convergence rates for $u_h$ in $L^2(H_0^1(\Omega))\cap H^1(L^2(\Omega))$ norm and for $v_h$ in ${L^2(H_0^1(\Omega))}$ norm. This illustrates that for sufficiently smooth exact solution, the numerical scheme works even for non-convex domains.  
\begin{figure}[H]
\begin{subfigure}{0.4\textwidth}
    \includegraphics[height=6cm,width=8cm]{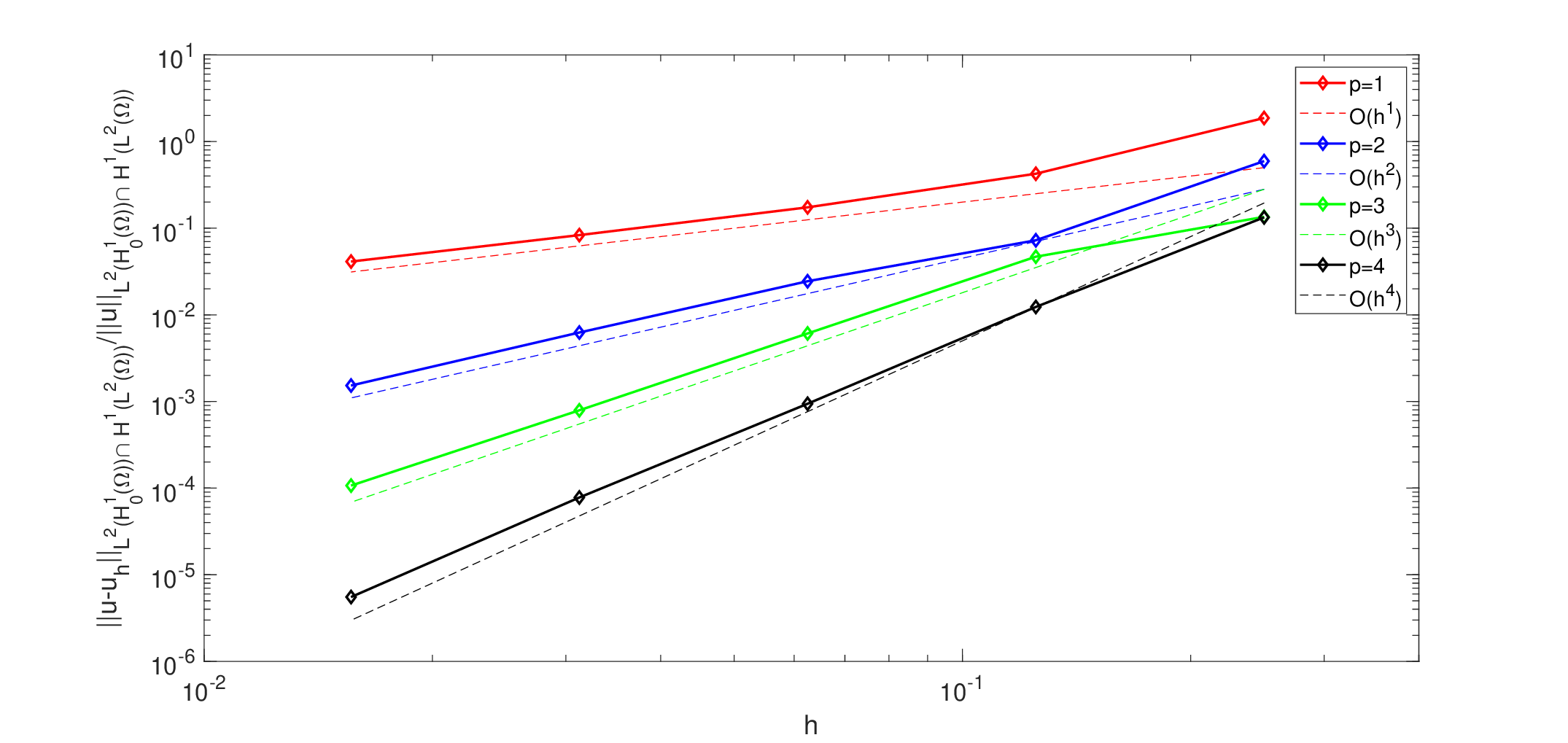}
    \caption[Caption]{}
    \label{fig:ring_u_H1}
\end{subfigure}
\hspace{1cm}
\begin{subfigure}{0.4\textwidth}
    \includegraphics[height=6cm,width=8cm]{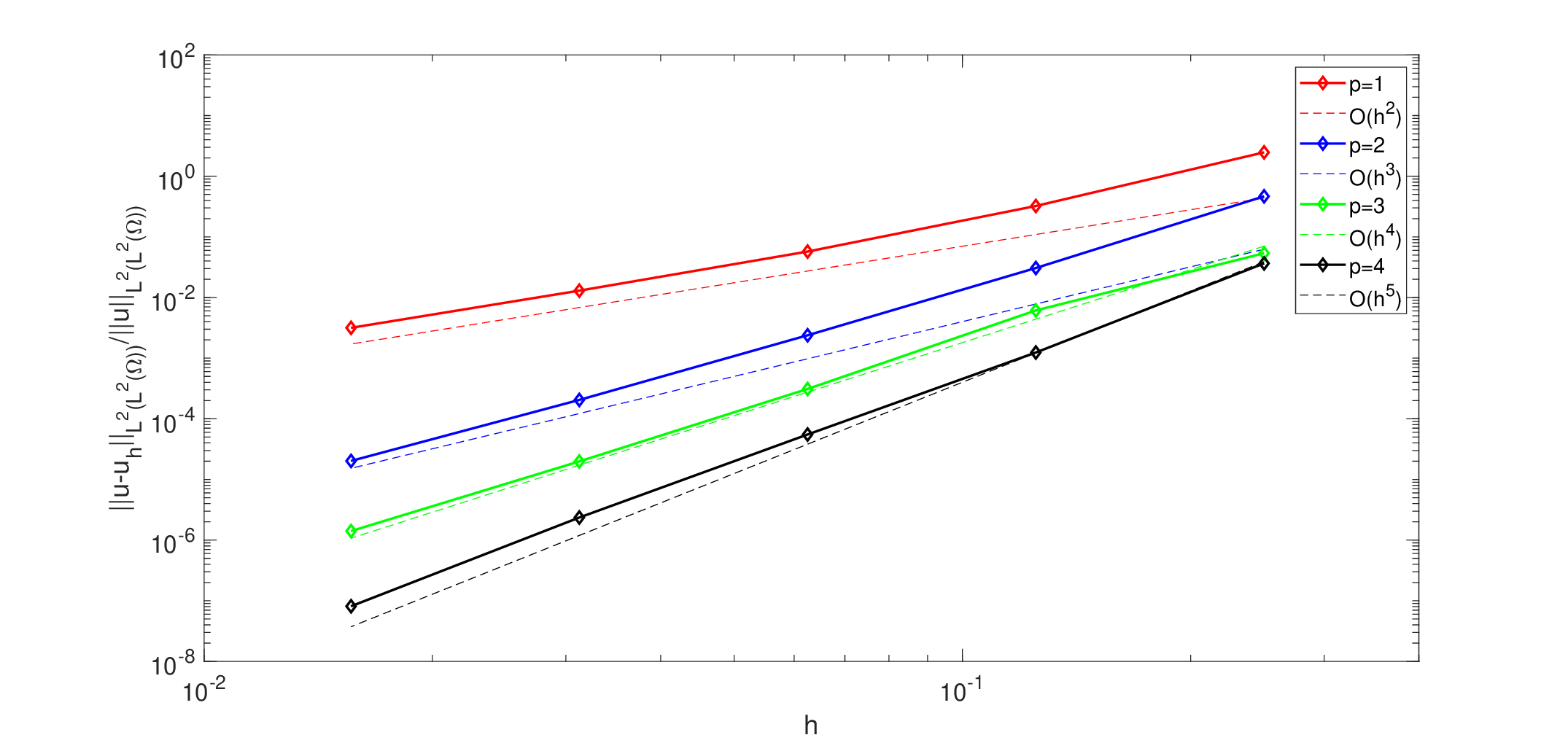}
     \caption[Caption]{}
    \label{fig:ring_u_L2}
\end{subfigure}
\caption{Relative errors $E_u^1$ and $E_u^2$ for ring shaped domain.}
\label{fig:ring_u}
\end{figure}

\begin{figure}[H]
\begin{subfigure}{0.4\textwidth}
    \includegraphics[height=6cm,width=8cm]{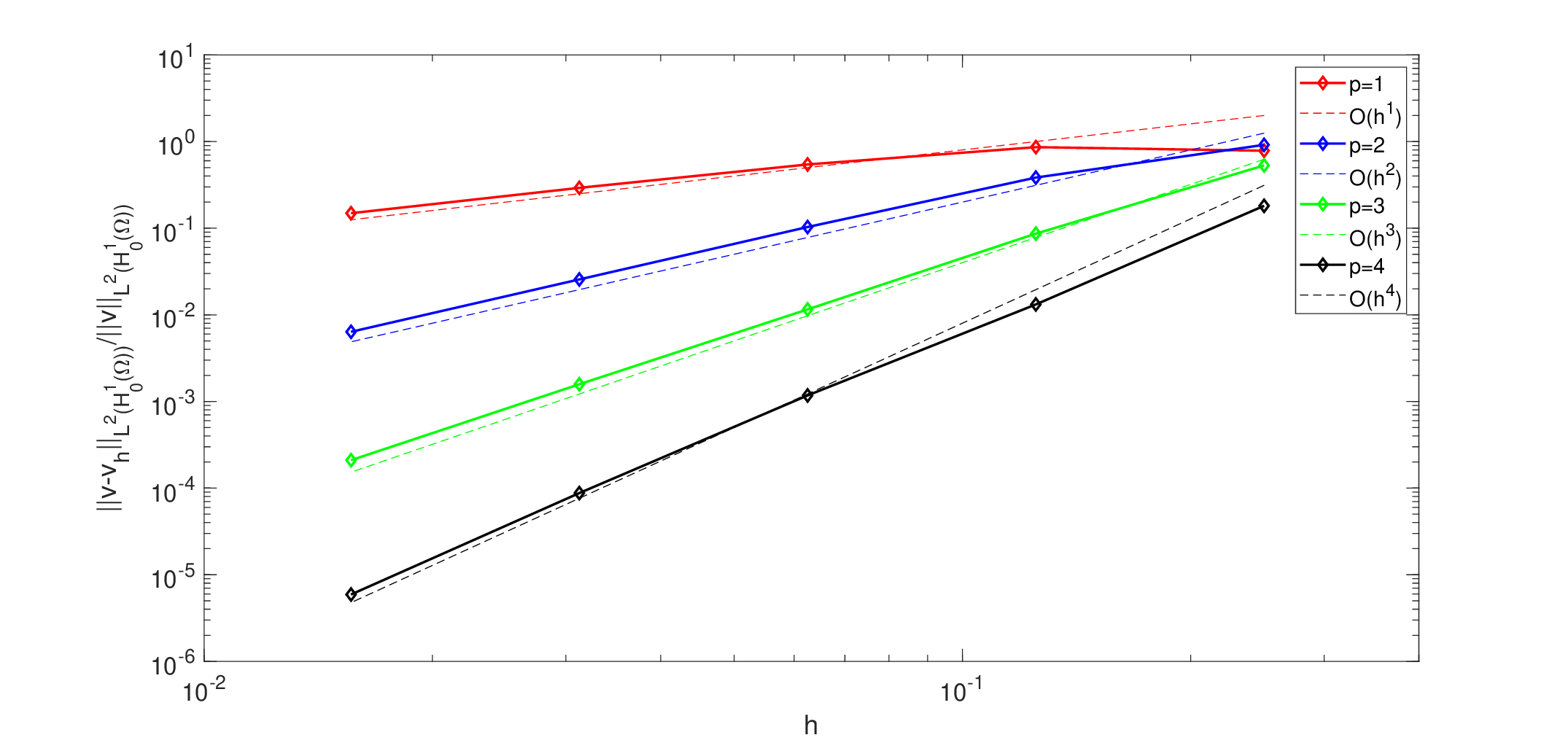}
    \caption[Caption]{}
    \label{fig:ring_v_H1}
\end{subfigure}
\hspace{1cm}
\begin{subfigure}{0.4\textwidth}
    \includegraphics[height=6cm,width=8cm]{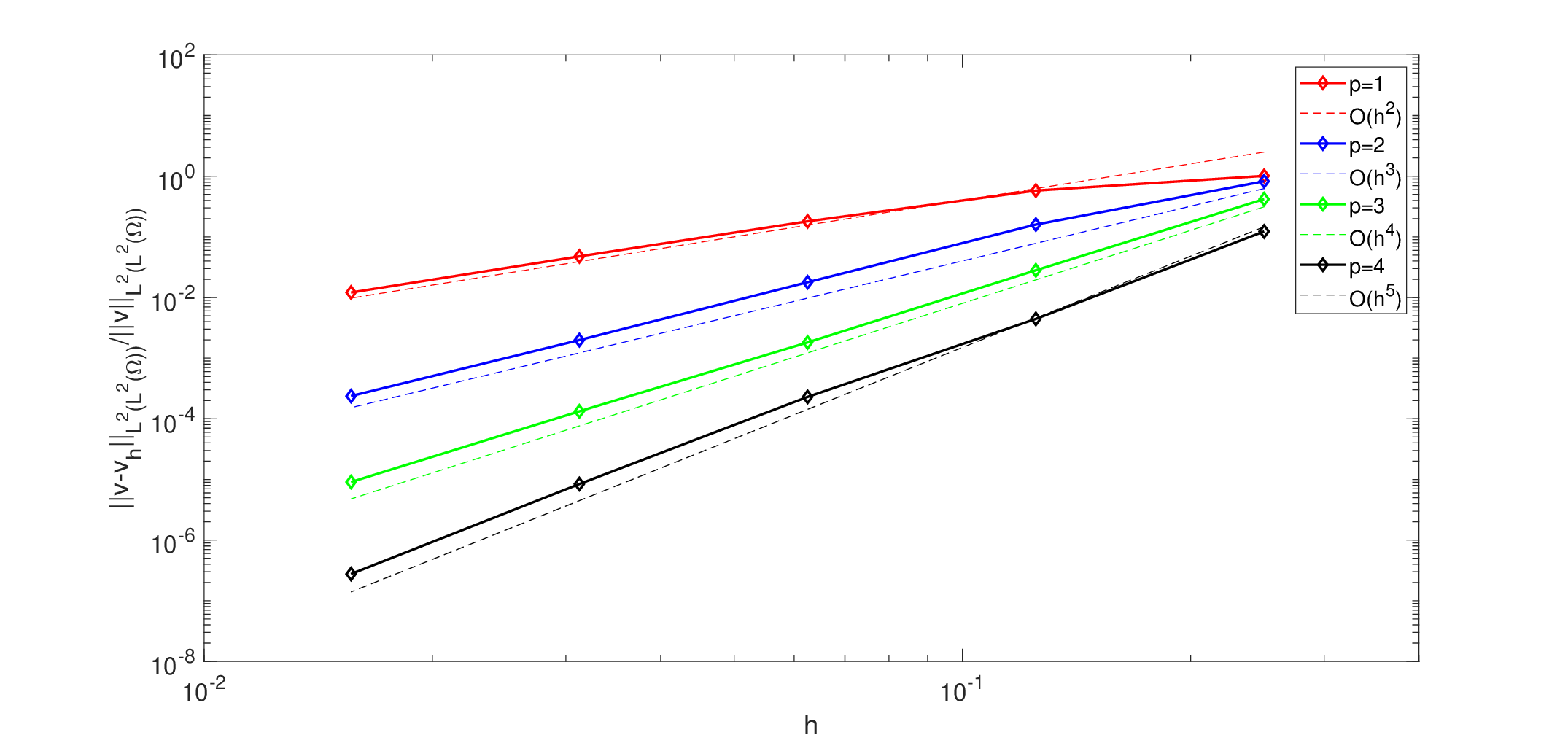}
    \caption[Caption]{}
    \label{fig:ring_v_L2}
\end{subfigure}
\caption{Relative errors $E_v^1$ and $E_v^2$ for ring shaped domain.}
\label{fig:ring_v}
\end{figure}



\section{Conclusions}\label{3sec5}
In the present work, we have analyzed the space-time isogeometric method for a linear fourth order time dependent problem. We have shown the unique solvability of the continuous and discrete variational formulations using the BNB theorem. The error estimates for the proposed method have been derived. Numerical results are provided which also confirm the theoretical findings. Further development is required to apply space-time IgA in the more general fourth order problems, like nonlinear EFK equation with other types of boundary conditions such as $u=\frac{\partial u}{\partial \nu}=0$ or $\frac{\partial u}{\partial \nu}=\frac{\partial \Delta u}{\partial \nu}=0$, where $\nu$ is the outward unit normal. Additionally, space-time formulation poses challenges in computational cost. Hence, the development of a suitable solver  with efficient preconditioning techniques for the linear system \eqref{3linsys} is also a desirable future work. \\



{\noindent{\bf Acknowledgment:}} The authors would like to acknowledge Dr. Debayan Maity for his insightful discussion and suggestions especially in the space-time weak formulation \eqref{3pro1bf} given in Section 2. The first author would like to acknowledge the financial support received from the Department of Science and Technology (DST), New Delhi, India through INSPIRE Fellowship.


\end{document}